\documentclass[11pt,a4paper]{amsart}
\usepackage[foot]{amsaddr}
\usepackage{amsmath,amssymb,amsthm}
\usepackage{mathrsfs}
\usepackage{tikz-cd}
\usepackage{enumitem}
\usepackage[utf8]{inputenc}
\usepackage{hyperref}
\usepackage{cleveref}
\usepackage{float}

\newcommand{\Syl}{\operatorname{Syl}}
\newcommand{\Enc}{\operatorname{Enc}}

\theoremstyle{definition}
\newtheorem{definition}{Definition}[section]
\newtheorem{question}[definition]{Question}
\newtheorem{example}[definition]{Example}

\theoremstyle{plain}
\newtheorem{theorem}[definition]{Theorem}
\newtheorem{proposition}[definition]{Proposition}
\newtheorem{lemma}[definition]{Lemma}
\newtheorem{corollary}[definition]{Corollary}

\theoremstyle{remark}

\numberwithin{equation}{section}

\title{Tree Bricks and Finite Tree Automata}
\author{Annoy Sengupta}
\address{}
\email{sengupta.annoy44@gmail.com}
\date{}

\begin{document}

\begin{abstract}
Let $\Lambda=KQ/I$ be a finite-dimensional zero-relation algebra. We encode Crawley--Boevey tree modules over $\Lambda$ by finite rooted trees labelled by arrows of $Q$ and their formal inverses, and construct a deterministic finite bottom-up tree automaton recognizing exactly these encodings. We define an accepted tree to be an automata-induced tree brick when it has no non-trivial factor--image self-overlap, and use Crawley--Boevey's graph-map basis to prove that this is equivalent to brickness of the associated tree module. We also introduce local colourings of $Q_1$ and show that the arrow alphabet can be compressed without changing the tree data, graph maps, or brick property. The optimal number of colours for such a compression is the maximum of the in-degree and out-degree of $Q$. We conclude by asking whether the tree language consisting only of bricks is regular.
\end{abstract}

\maketitle

\section{Introduction}
\label{sec:introduction}

Let $K$ be an algebraically closed field and let $\Lambda=KQ/I$ be a finite-dimensional zero-relation algebra. Crawley--Boevey introduced tree modules over zero-relation algebras in \cite{CB89}. Such a module is obtained from a finite tree $T$ and a bound quiver morphism $F:T\to Q$ satisfying a local injectivity condition. Tree modules are indecomposable, a fact essentially coming from Gabriel's universal-covering methods \cite{Gab81}, and Crawley--Boevey \cite{CB89} gave an explicit combinatorial basis for the Hom-space between two tree modules. The basis elements, called graph maps, are controlled by matching a factor subtree in the source with an image subtree in the target. In particular, the endomorphism algebra of a tree module is determined by self-overlaps of this kind.

Tree modules are particularly well suited to an automata-theoretic treatment because both their construction and their morphisms are governed by finite combinatorial data. This raises a natural structural question: can the class of tree modules, together with the combinatorics relevant to their endomorphisms, be organized within a finite-state framework? The purpose of this paper is to show that the answer is affirmative at the level of tree data and to use this framework to formulate brickness directly on the accepted trees.

There is a well-established interaction between automata theory and representation theory of bound quivers. Rees \cite{Rees08} showed that the strings of a monomial algebra form a locally testable, hence regular, language and described a finite automaton recognizing them. Srivastava and Kuber \cite{SrivastavaKuber2025} used finite automata to obtain algorithmic descriptions of the linear orders arising from hammocks of special biserial algebras, providing automata-theoretic proofs of key ingredients in the computation of their stable rank. More recently, Kuber and Sengupta introduced multi-entry inverse automata for string algebras and translated the factor--image criterion for string bricks into the language of accepted pointed words; see \cite[Theorems~5.6 and~5.10]{KuberSengupta2026}. Their construction also admits an alphabet reduction through local bijections \cite[Definition~4.11 and Corollary~4.13]{KuberSengupta2026}.

For tree modules the underlying combinatorial object is no longer a word but a branching tree, so finite word automata are replaced naturally by finite tree automata. The foundations of finite tree automata go back to Thatcher--Wright \cite{ThatcherWright68} and Doner \cite{Doner70}; we use the standard bottom-up formalism of \cite{TATA}. We first choose a root in $T$ and record each edge by an arrow of $Q$ or its formal inverse, according to its direction relative to the root. The resulting rooted labelled tree determines the original tree datum uniquely, up to the harmless choice of root. We then construct in \S~\ref{subsec:the automaton} a deterministic finite bottom-up automaton recognizing exactly the rooted encodings of tree data. The only non-local issue is the avoidance of the zero-relations. Since the defining relations have bounded length, this can be checked with finite memory by storing bounded incoming and outgoing path profiles at the root of each processed subtree.

We next introduce (\Cref{def:automata-tree-brick}) a notion of brickness directly on the trees accepted by $\mathcal A_\Lambda$. In the rooted syllable language, factor and image subtrees are detected by the orientations of their boundary syllables, and a non-trivial factor--image self-overlap is given by two such subtrees carrying the same typed syllable-labelled tree structure. We call an accepted tree with no such non-trivial self-overlap an \emph{automata-induced tree brick}. Using Crawley--Boevey's graph-map basis, we then show (\Cref{thm:brick-criterion}) that this automata-theoretic notion agrees exactly with the usual representation-theoretic one: an accepted tree is an automata-induced tree brick if and only if the associated Crawley--Boevey tree module has endomorphism ring $K$. Thus the automaton provides a branching counterpart of the factor--image description of brickness familiar from string modules.

Finally, we show (\Cref{thm:optimal-colouring}) that the full alphabet $Q_1\sqcup Q_1^{-1}$ is larger than necessary. A local colouring is a colouring of $Q_1$ which is injective separately on arrows with a common source and on arrows with a common target. Because the vertex type is retained in our encoding, the type together with the signed colour uniquely recovers the original arrow. Thus local colourings preserve the entire tree datum, and consequently preserve graph maps and brickness. By translating local colourings into edge-colourings of a bipartite multigraph and applying K\"onig's line-colouring theorem, we determine the optimal number of colours. We end with one question (Question \ref{que:regular-bricks}) left open by the paper: whether the subclass of accepted trees corresponding to bricks is itself a regular tree language.

The paper is organized as follows. In \S~\ref{sec:tree-data} we recall tree modules and encode tree data by rooted syllable trees. In \S~\ref{sec:automaton} we construct the finite tree automaton and prove the recognition theorem. In  \S~\ref{sec:bricks} we introduce automata-induced tree bricks and prove their equivalence with brick Crawley--Boevey tree modules, together with explicit brick and non-brick examples. In \S~\ref{sec:compression} we develop alphabet compression and determine the optimal colour alphabet. In \S~\ref{sec:regularity} we formulate the regularity problem for the brick language.

\section{Tree data and rooted syllable trees}
\label{sec:tree-data}

Throughout the paper, $Q=(Q_0,Q_1,s,t)$ is a finite quiver. Let $J\subset KQ$ be the arrow ideal and let $\rho$ be a finite set of paths of length at least two such that $I:=\langle\rho\rangle$ is admissible. Thus $J^N\subseteq I\subseteq J^2$ for some $N\geq 2$, and $\Lambda:=KQ/I$ is a finite-dimensional zero-relation algebra. We write a path as $p=\alpha_m\cdots\alpha_1$, where $\alpha_1$ is traversed first; its length is denoted by $|p|$. A path belongs to $I$ precisely when it contains an element of $\rho$ as a consecutive subpath.

A \emph{tree} is a finite quiver whose underlying undirected graph is a tree. If $T$ is a tree, let $V_T$ be the representation of $T$ which is $K$ at every vertex and whose arrow maps are all $1_K$.

\begin{definition}\label{def:tree-datum}
A \emph{$\Lambda$-tree datum} is a pair $(T,F)$ consisting of a tree $T$ and a quiver morphism $F:T\to Q$ satisfying:
\begin{enumerate}[label=\textup{(\roman*)}]
\item no path $p$ in $T$ satisfies $F(p)\in I$;
\item if $a,b\in T_1$ are distinct and have a common source or a common target, then $F(a)\neq F(b)$.
\end{enumerate}
The associated \emph{tree module} is $M(T,F):=F_\lambda(V_T)$.
\end{definition}

Explicitly, $M(T,F)_i=\bigoplus_{F(x)=i}Kv_x$ for $i\in Q_0$, and for $\alpha:i\to j$ one has
\[
M(T,F)_\alpha(v_x)=\begin{cases}v_y,&\text{if there is an arrow }a:x\to y\text{ in }T\text{ with }F(a)=\alpha,\\0,&\text{otherwise.}\end{cases}
\]
Condition (ii) makes the displayed formula unambiguous. This is Crawley--Boevey's definition \cite[\S 1]{CB89}. We will use his description of Hom-spaces in \S~\ref{sec:bricks}.

\subsection{Rooted syllable encoding}

For each $\alpha\in Q_1$, introduce a formal inverse $\alpha^{-1}$ and set $Q_1^{\pm}:=Q_1\sqcup Q_1^{-1}$. Define $s(\alpha^{-1})=t(\alpha)$, $t(\alpha^{-1})=s(\alpha)$ and $(\alpha^{-1})^{-1}=\alpha$. For $i\in Q_0$ put
\[
\Syl(i):=\{\gamma\in Q_1^{\pm}\mid s(\gamma)=i\}.
\]

Let $(T,F)$ be a tree datum. For adjacent vertices $x,y\in T_0$, define the intrinsic syllable $\lambda_F(x,y)\in Q_1^{\pm}$ by
\[
\lambda_F(x,y)=\begin{cases}F(a),&a:x\to y,\\F(a)^{-1},&a:y\to x.\end{cases}
\]
Then $\lambda_F(y,x)=\lambda_F(x,y)^{-1}$, $s(\lambda_F(x,y))=F(x)$ and $t(\lambda_F(x,y))=F(y)$.

Choose a root $r\in T_0$. Every edge now has a unique parent--child orientation. We record the underlying rooted undirected tree, the type map $F:T_0\to Q_0$, and on every parent--child edge $(x,y)$ the label $\lambda_F(x,y)$.

It is convenient to describe directly the rooted objects which arise in this way. Let $S$ be a finite rooted undirected tree with root $r$, let $\tau:S_0\to Q_0$ be a type map, and label each parent--child edge $(x,y)$ by a syllable $\ell(x,y)\in Q_1^{\pm}$. Extend the labelling to both orientations by $\widetilde\ell(y,x)=\widetilde\ell(x,y)^{-1}$.

\begin{definition}\label{def:admissible-syllable-tree}
The rooted syllable tree $\mathcal S=(S,r,\tau,\ell)$ is \emph{$\Lambda$-admissible} if:
\begin{enumerate}[label=\textup{(\roman*)}]
\item $\mathcal S$ is \emph{type compatible}, i.e. for every parent--child edge $(x,y)$, $s(\ell(x,y))=\tau(x)$ and $t(\ell(x,y))=\tau(y)$;
\item $\mathcal S$ is \emph{locally injective}, i.e. for every $x\in S_0$, the map $N_S(x)\to Q_1^{\pm}$, $y\mapsto\widetilde\ell(x,y)$, is injective;
\item $\mathcal S$ is  \emph{relation admissible}, i.e. no simple path $(x_0,\ldots,x_m)$ for which all $\widetilde\ell(x_{j-1},x_j)$ lie in $Q_1$ has $Q$-label in $I$.
\end{enumerate}
\end{definition}

Condition (ii) is exactly the tree-module condition in Definition~\ref{def:tree-datum}: direct syllables at $x$ correspond to arrows of $T$ leaving $x$, while inverse syllables correspond to arrows entering $x$.

\begin{theorem}[Rooted encoding]\label{thm:rooted-encoding}
Rooted $\Lambda$-tree data $(T,F,r)$ are in bijection with $\Lambda$-admissible rooted syllable trees. Under this bijection the underlying undirected tree and the type map are unchanged.
\end{theorem}

\begin{proof}
The construction from $(T,F,r)$ was given above. Conditions (i) and (ii) follow from the fact that $F$ is a quiver morphism satisfying the tree-module condition. A directed path in $T$ is precisely a simple path whose syllables, when traversed in the direction of the path, all lie in $Q_1$; hence boundness of $F$ is equivalent to condition (iii).

Conversely, let $\mathcal S=(S,r,\tau,\ell)$ be admissible. For every undirected edge $\{x,y\}$, exactly one of $\widetilde\ell(x,y)$ and $\widetilde\ell(y,x)$ lies in $Q_1$. Orient the edge from $x$ to $y$ when $\widetilde\ell(x,y)\in Q_1$, and label the resulting arrow by $\widetilde\ell(x,y)$. This produces a quiver $T$ and a quiver morphism $F:T\to Q$ with $F|_{T_0}=\tau$. Conditions (ii) and (iii) give respectively the tree-module condition and boundness. The two constructions are inverse.
\end{proof}

Changing the root does not change $\tau$ or the intrinsic labels $\widetilde\ell(x,y)$; it only changes which orientation of an edge is declared parent--child. We call this operation \emph{rerooting}. Theorem~\ref{thm:rooted-encoding} immediately gives the following unrooted form.

\begin{corollary}\label{cor:rerooting}
Isomorphism classes of $\Lambda$-tree data are in bijection with equivalence classes of $\Lambda$-admissible rooted syllable trees, where two rooted syllable trees are equivalent if one is isomorphic to a rerooting of the other.
\end{corollary}

Fix from now on a total order $\prec$ on the finite set $Q_1^{\pm}$. In an admissible rooted syllable tree the labels from a vertex to its children are distinct, so we order the children increasingly by these labels. Hence every admissible rooted syllable tree becomes canonically an ordered rooted tree. Moreover, if
\[
D_Q:=\max_{i\in Q_0}|\Syl(i)|,
\]
then every vertex has degree at most $D_Q$. These two elementary observations allow us to work with finite ranked tree automata in the next section.

\section{A finite tree automaton for tree data}
\label{sec:automaton}

We use the standard bottom-up model of finite tree automata; see \cite{TATA}. A finite ranked alphabet is a finite disjoint union $\Sigma=\coprod_{m\geq 0}\Sigma_m$, where a symbol in $\Sigma_m$ has rank $m$. A deterministic finite bottom-up tree automaton is a tuple $(\mathcal Q,\Sigma,\delta,\mathcal F)$ with finite state set $\mathcal Q$, final states $\mathcal F\subseteq\mathcal Q$, and transition maps $\delta_m:\Sigma_m\times\mathcal Q^m\to\mathcal Q$. A finite $\Sigma$-tree is accepted when the state computed at its root lies in $\mathcal F$.

\subsection{The input alphabet and path profiles}

For $0\leq m\leq D_Q$, let $\Sigma_{\Lambda,m}$ consist of the symbols
\[
[i;\gamma_1,\ldots,\gamma_m]
\]
with $i\in Q_0$, $\gamma_j\in\Syl(i)$ and $\gamma_1\prec\cdots\prec\gamma_m$. Put $\Sigma_\Lambda:=\coprod_{m=0}^{D_Q}\Sigma_{\Lambda,m}$. The symbol records the type of a vertex and the ordered syllables leading to its children. Thus an admissible rooted syllable tree $\mathcal S$ determines a ranked tree $\Enc(\mathcal S)\in\mathsf T_{\Sigma_\Lambda}$.

The automaton must check two kinds of compatibility when child subtrees are attached to a new root: the type and local-injectivity conditions, and the absence of a relation in $\rho$. The latter is the only point requiring memory beyond the labels adjacent to the root.

Set
\[
L:=\max\bigl(\{|r|:r\in\rho\}\cup\{1\}\bigr).
\]
Let $\mathsf P_{<L}(Q)$ be the finite set of paths in $Q$ of length $<L$, including the trivial paths $e_i$. For $i\in Q_0$ write
\[
\mathsf P^{\mathrm{in}}_{<L}(i):=\{p\in\mathsf P_{<L}(Q):t(p)=i\},\qquad \mathsf P^{\mathrm{out}}_{<L}(i):=\{p\in\mathsf P_{<L}(Q):s(p)=i\}.
\]
For a rooted syllable tree $\mathcal S$ with root $r$ of type $i$, let $\mathsf I_{<L}(\mathcal S)$ be $\{e_i\}$ together with the labels of all non-trivial directed paths of length $<L$ ending at $r$, and define $\mathsf O_{<L}(\mathcal S)$ dually using directed paths starting at $r$. We use $\mathsf I_{\leq L}$ and $\mathsf O_{\leq L}$ with the evident meaning.

\begin{lemma}[Profiles at a new root]\label{lem:profiles}
Let $\mathcal S$ have root $r$ of type $i$, children $y_1,\ldots,y_m$, and edge labels $\gamma_j=\ell(r,y_j)$. Let $\mathcal S_j$ be the subtree rooted at $y_j$. Then
\[
\mathsf I_{\leq L}(\mathcal S)=\{e_i\}\cup\bigcup_{\gamma_j\in Q_1^{-1}}\{\gamma_j^{-1}p:p\in\mathsf I_{<L}(\mathcal S_j),\ |p|+1\leq L\},
\]
\[
\mathsf O_{\leq L}(\mathcal S)=\{e_i\}\cup\bigcup_{\gamma_j\in Q_1}\{p\gamma_j:p\in\mathsf O_{<L}(\mathcal S_j),\ |p|+1\leq L\}.
\]
If every $\mathcal S_j$ is relation admissible, then $\mathcal S$ is relation admissible if and only if there are no $r_0\in\rho$, $u\in\mathsf I_{\leq L}(\mathcal S)$ and $v\in\mathsf O_{\leq L}(\mathcal S)$ with $r_0=vu$.
\end{lemma}

\begin{proof}
A non-trivial directed path ending at $r$ must use, as its last edge, a unique child edge labelled $\gamma_j^{-1}$; deleting that edge leaves a directed path ending at $y_j$. This gives the first formula. The second is dual. Now assume the child subtrees contain no relation. Any occurrence of a generator $r_0\in\rho$ in $\mathcal S$ which is not contained in one child subtree must pass through $r$. Splitting it at $r$ gives an incoming part $u$ and an outgoing part $v$, with $r_0=vu$. Conversely, such a pair $u,v$ is realized on two branches which meet only at $r$ (unless one part is trivial), and hence gives a directed simple path labelled by $r_0$. Since $I$ is generated by $\rho$, avoiding the generators is equivalent to avoiding all paths in $I$.
\end{proof}

\subsection{The automaton}
\label{subsec:the automaton}
For $i\in Q_0$, let $\mathcal Q_i$ consist of all tuples $(i,C,\mathsf I,\mathsf O)$ with $C\subseteq\Syl(i)$, $\mathsf I\subseteq\mathsf P^{\mathrm{in}}_{<L}(i)$, $\mathsf O\subseteq\mathsf P^{\mathrm{out}}_{<L}(i)$, and $e_i\in\mathsf I\cap\mathsf O$. Add a failure state $\bot$ and set $\mathcal Q_\Lambda:=\{\bot\}\cup\bigcup_{i\in Q_0}\mathcal Q_i$. The component $C$ records the syllables from the root to its children; it is needed when this subtree is later attached to a parent.

Let $\sigma=[i;\gamma_1,\ldots,\gamma_m]$ and suppose the child states are $q_j=(i_j,C_j,\mathsf I_j,\mathsf O_j)$. If some child state is $\bot$, if $i_j\neq t(\gamma_j)$ for some $j$, or if $\gamma_j^{-1}\in C_j$ for some $j$, define the transition to be $\bot$. Now check for type compatibility and local injectivity at the child root after the parent edge is attached. Assume these tests pass and define
\begin{align*}
\widehat{\mathsf I}&:=\{e_i\}\cup\bigcup_{\gamma_j\in Q_1^{-1}}\{\gamma_j^{-1}p:p\in\mathsf I_j,\ |p|+1\leq L\},\\
\widehat{\mathsf O}&:=\{e_i\}\cup\bigcup_{\gamma_j\in Q_1}\{p\gamma_j:p\in\mathsf O_j,\ |p|+1\leq L\}.
\end{align*}
If there exist $r_0\in\rho$, $u\in\widehat{\mathsf I}$ and $v\in\widehat{\mathsf O}$ with $r_0=vu$, again output $\bot$. Otherwise set
\[
\delta_m(\sigma,q_1,\ldots,q_m):=\bigl(i,\{\gamma_1,\ldots,\gamma_m\},\widehat{\mathsf I}\cap\mathsf P_{<L}(Q),\widehat{\mathsf O}\cap\mathsf P_{<L}(Q)\bigr).
\]
For $m=0$ this gives $\delta_0([i;])=(i,\varnothing,\{e_i\},\{e_i\})$. Finally put $\mathcal F_\Lambda:=\mathcal Q_\Lambda\setminus\{\bot\}$ and
\[
\mathcal A_\Lambda:=(\mathcal Q_\Lambda,\Sigma_\Lambda,\delta,\mathcal F_\Lambda).
\]
All sets involved are finite, so $\mathcal A_\Lambda$ is a deterministic finite bottom-up tree automaton.

\begin{theorem}[Recognition theorem]\label{thm:recognition}
The language $L(\mathcal A_\Lambda)$ consists exactly of the ranked encodings $\Enc(\mathcal S)$ of $\Lambda$-admissible rooted syllable trees. More precisely, if $\mathcal S$ has root of type $i$, then the state at its root is
\[
\bigl(i,C(\mathcal S),\mathsf I_{<L}(\mathcal S),\mathsf O_{<L}(\mathcal S)\bigr),
\]
where $C(\mathcal S)$ is the set of labels from the root to its children.
\end{theorem}

\begin{proof}
We proceed by induction on the height of the input tree. If the input tree consists of a single vertex labelled by $[i;]$, then the transition produces the state $(i,\varnothing,\{e_i\},\{e_i\}),$ which is exactly the state attached to the one-vertex $\Lambda$-admissible rooted syllable tree of type $i$. Thus the statement holds in height zero.

Now suppose the theorem holds for all trees of height strictly smaller than that of $\mathbf t=[i;\gamma_1,\ldots,\gamma_m](\mathbf t_1,\ldots,\mathbf t_m).$ Assume first that the transition at the root does not produce the failure state. By the induction hypothesis, each child tree $\mathbf t_j$ is the encoding of a $\Lambda$-admissible rooted syllable tree $\mathcal S_j$, and its state is $(i_j,C_j,\mathsf I_j,\mathsf O_j),$ where $i_j$ is the type of the root of $\mathcal S_j$, $C_j$ is the set of syllables on the edges from that root to its children, and $\mathsf I_j,\mathsf O_j$ are its incoming and outgoing profiles of length $<L$. Attach the roots of the trees $\mathcal S_1,\ldots,\mathcal S_m$ to a new root $r$ of type $i$, labelling the edge from $r$ to the root of $\mathcal S_j$ by $\gamma_j$. Since the transition is non-failing, one has $i_j=t(\gamma_j)$ for every $j$. Together with $s(\gamma_j)=i$, which is built into the symbol $[i;\gamma_1,\ldots,\gamma_m]$, this shows that the resulting rooted syllable tree is type compatible.

We next verify local injectivity. The strict ordering $\gamma_1\prec\cdots\prec\gamma_m$ implies that the syllables on the edges from the new root to its children are pairwise distinct, so local injectivity holds at the new root. At the root of the $j$-th child subtree, the newly attached parent edge is seen with syllable $\gamma_j^{-1}$. Since the transition requires $\gamma_j^{-1}\notin C_j,$ this new incident syllable is distinct from all syllables already occurring on edges from that child root to its own children. Local injectivity at all other vertices holds by the induction hypothesis. Hence the enlarged tree is locally injective.

By Lemma~\ref{lem:profiles}, the sets $\widehat{\mathsf I}$ and $\widehat{\mathsf O}$ computed in the transition are precisely the incoming and outgoing directed-path profiles of length at most $L$ at the new root. Since each child subtree is already relation admissible, any new forbidden relation must pass through the new root. Again by Lemma~\ref{lem:profiles}, such a relation occurs if and only if there exist $
r_0\in\rho,\ 
u\in\widehat{\mathsf I},\ 
v\in\widehat{\mathsf O}$ with $r_0=vu.$ The transition is assumed not to fail, so no such triple exists. Therefore the enlarged rooted syllable tree is relation-admissible, and hence $\Lambda$-admissible.

Finally, Lemma~\ref{lem:profiles} shows that, after discarding paths of length exactly $L$, the sets $
\mathsf I=\widehat{\mathsf I}\cap\mathsf P_{<L}(Q)$ and $
\mathsf O=\widehat{\mathsf O}\cap\mathsf P_{<L}(Q) $ are exactly the incoming and outgoing profiles of length $<L$ at the new root. The set $C=\{\gamma_1,\ldots,\gamma_m\}$ is precisely the set of child syllables at that root. Thus the state produced by the transition is the state asserted in the theorem.

Conversely, suppose that $\mathbf t$ is the encoding of a $\Lambda$-admissible rooted syllable tree. Each child subtree is again $\Lambda$-admissible, so by the induction hypothesis its state is non-failing and records the correct root type, child syllables, and truncated incoming and outgoing profiles. Type-compatibility gives $i_j=t(\gamma_j)$ for every $j$. Local injectivity at the root of the $j$-th child subtree implies $\gamma_j^{-1}\notin C_j,$ and the child syllables $\gamma_1,\ldots,\gamma_m$ are pairwise distinct. Finally, relation admissibility of the whole tree and Lemma~\ref{lem:profiles} imply that the obstruction condition $r_0=vu$ never occurs. Hence the transition does not fail. The same profile decomposition shows that the resulting state is exactly the one prescribed in the statement of the theorem.
\end{proof}

Combining Theorem~\ref{thm:recognition} with Corollary~\ref{cor:rerooting}, tree data are therefore represented by rerooting classes inside a regular tree language.

\begin{example}[Why incoming and outgoing profiles are needed]\label{ex:relation-profile}
Consider the quiver $Q$ and rooted tree $T$ displayed in Figure~\ref{fig:relation-profile}. Define $F(u)=1$, $F(x)=2$, $F(v)=3$, $F(w)=4$, $F(a)=\alpha$, $F(b)=\beta$ and $F(c)=\gamma$, and let $I=\langle\gamma\alpha\rangle$. Since $Q$ is acyclic, this is an admissible monomial ideal. Relative to the root $x$, the three child labels are $\alpha^{-1}$, $\beta^{-1}$ and $\gamma$. The $\alpha^{-1}$-branch contributes the incoming path $\alpha$, while the $\gamma$-branch contributes the outgoing path $\gamma$. At the root the automaton therefore detects $\gamma\alpha\in I$ and rejects the tree. Neither offending piece lies inside a single child subtree; the relation appears only after the branches are joined.
\end{example}

\begin{figure}[H]
\centering
\begin{minipage}{0.43\textwidth}
\centering
\begin{tikzcd}[column sep=large,row sep=large]
1 \arrow[r,"\alpha"] & 2 \arrow[r,"\gamma"] & 4 \\
3 \arrow[ur,"\beta"'] & &
\end{tikzcd}
\\
$Q$
\end{minipage}
\hfill
\begin{minipage}{0.43\textwidth}
\centering
\begin{tikzcd}[column sep=large,row sep=large]
u \arrow[r,"a"] & x \arrow[r,"c"] & w \\
v \arrow[ur,"b"'] & &
\end{tikzcd}
\\
$T$, rooted at $x$
\end{minipage}
\caption{A relation created by joining two branches at the root.}
\label{fig:relation-profile}
\end{figure}
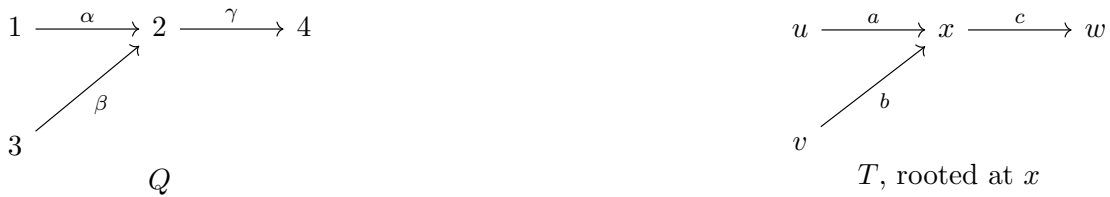

\section{Automata-induced tree bricks}
\label{sec:bricks}

Theorem~\ref{thm:recognition} identifies tree data over $\Lambda$ with rooted labelled trees accepted by $\mathcal A_\Lambda$. We now define brickness directly on such an accepted tree. The definition uses only the vertex types, the intrinsic syllable labelling, and the orientations of boundary edges. We then show that it agrees exactly with brickness of the associated Crawley--Boevey tree module.

We stress that the terminology \emph{automata-induced tree brick} does not mean that the brick trees are already known to form a regular tree language. It means only that brickness is formulated intrinsically on the objects accepted by $\mathcal A_\Lambda$. Whether this subclass is itself regular will be considered in \S~\ref{sec:regularity}.

\subsection{Factor--image overlaps in an accepted tree}

Let $\mathbf t\in L(\mathcal A_\Lambda)$. By Theorem~\ref{thm:recognition}, $\mathbf t$ is the ranked encoding of a $\Lambda$-admissible rooted syllable tree
$
\mathcal S_{\mathbf t}=(S,r,\tau,\ell).
$
Let $\tau:S_0\to Q_0$ be its type map and let
$
\widetilde{\ell}:\{(x,y)\mid x,y\text{ are adjacent in }S\}\longrightarrow Q_1^{\pm}
$
be the intrinsic syllable labelling introduced in \S~\ref{sec:tree-data}. Recall that $\widetilde{\ell}(y,x)=\widetilde{\ell}(x,y)^{-1}$.

By a subtree $U\subseteq S$ we mean a non-empty connected induced subtree. Define its boundary, viewed from inside $U$, by
$
\partial_S U:=\{(x,y)\mid x\in U_0,\ y\notin U_0,\ \{x,y\}\text{ is an edge of }S\}.
$

\begin{definition}\label{def:automata-factor-image}
Let $\mathbf t\in L(\mathcal A_\Lambda)$ and let $U\subseteq S$ be a subtree.

\begin{enumerate}[label=\textup{(\roman*)}]
\item We call $U$ an \emph{automata-factor subtree} if $\widetilde{\ell}(x,y)\in Q_1$ for every $(x,y)\in\partial_S U$.
\item We call $U$ an \emph{automata-image subtree} if $\widetilde{\ell}(x,y)\in Q_1^{-1}$ for every $(x,y)\in\partial_S U$.
\end{enumerate}
\end{definition}

Thus every boundary edge of an automata-factor subtree is directed, in the quiver encoded by $\mathcal S_{\mathbf t}$, from the subtree towards its complement, whereas every boundary edge of an automata-image subtree is directed towards the subtree.

We next specify when two subtrees carry the same labelled-tree structure.

\begin{definition}\label{def:automata-labelled-isomorphism}
Let $U,V\subseteq S$ be subtrees. An isomorphism of underlying trees
$
\Phi:U\xrightarrow{\sim}V
$
is called \emph{automata-label-preserving} if
$
\tau(\Phi(x))=\tau(x)
$
for every $x\in U_0$, and
$
\widetilde{\ell}(\Phi(x),\Phi(y))
=
\widetilde{\ell}(x,y)
$
for every ordered pair of adjacent vertices $x,y\in U_0$.
\end{definition}

\begin{definition}\label{def:automata-self-overlap}
An \emph{automata-induced self-overlap} of $\mathbf t$ is a triple
$
(U,V,\Phi)
$
such that $U$ is an automata-factor subtree, $V$ is an automata-image subtree, and $\Phi:U\xrightarrow{\sim}V$ is automata-label-preserving.

The overlap $(S,S,\operatorname{id}_S)$ is called the \emph{trivial overlap}. We denote the set of automata-induced self-overlaps of $\mathbf t$ by
$\mathscr O_{\mathcal A}(\mathbf t).
$
\end{definition}

The whole tree has empty boundary and is therefore both an automata-factor and an automata-image subtree. Hence
$(S,S,\operatorname{id}_S)\in\mathscr O_{\mathcal A}(\mathbf t).$

\begin{definition}[Automata-induced tree brick]\label{def:automata-tree-brick}
An accepted tree $\mathbf t\in L(\mathcal A_\Lambda)$ is called an \emph{automata-induced tree brick} if
$\mathscr O_{\mathcal A}(\mathbf t)
=
\{(S,S,\operatorname{id}_S)\}.$
Equivalently, $\mathbf t$ is an automata-induced tree brick if it admits no non-trivial factor--image self-overlap.
\end{definition}

Although the definition is made after choosing a root, it is independent of that choice.

\begin{proposition}\label{prop:automata-brick-rerooting}
Let $\mathbf t$ and $\mathbf t'$ be the ranked encodings obtained from the same tree datum by choosing two different roots. Then $\mathbf t$ is an automata-induced tree brick if and only if $\mathbf t'$ is an automata-induced tree brick.
\end{proposition}

\begin{proof}
Rerooting changes which endpoint of an edge is regarded as its parent, but it does not change the type map $\tau$ or the intrinsic labelling $\widetilde{\ell}$. Consequently the boundary conditions in Definition~\ref{def:automata-factor-image} and the label-preserving condition in Definition~\ref{def:automata-labelled-isomorphism} are unchanged. Thus rerooting induces a canonical bijection between the two sets of self-overlaps.
\end{proof}

\subsection{Comparison with Crawley--Boevey graph maps}

Let $(T_{\mathbf t},F_{\mathbf t})$ be the tree datum reconstructed from $\mathcal S_{\mathbf t}$ by Theorem~\ref{thm:rooted-encoding}. Thus $T_{\mathbf t}$ has underlying undirected tree $S$, and an edge $\{x,y\}$ is oriented from $x$ to $y$ precisely when
$
\widetilde{\ell}(x,y)\in Q_1.
$
Moreover,
$
F_{\mathbf t}(x)=\tau(x),
$
and if $a:x\to y$ is the corresponding arrow of $T_{\mathbf t}$, then
$
F_{\mathbf t}(a)=\widetilde{\ell}(x,y).
$

Recall that Crawley--Boevey \cite[\S 2]{CB89} describes graph maps between two tree modules by a non-empty factor subtree of the source, an image subtree of the target, and a quiver isomorphism between them which is compatible with the maps to $Q$. These graph maps form a basis of the corresponding hom-space.

The following proposition gives the precise dictionary between these graph-map data and the automata-induced overlaps defined above.

\begin{proposition}[Automata--graph-map dictionary]\label{prop:automata-graph-map-dictionary}
Let $\mathbf t\in L(\mathcal A_\Lambda)$ and let $(T_{\mathbf t},F_{\mathbf t})$ be its reconstructed tree datum.

\begin{enumerate}[label=\textup{(\roman*)}]
\item A subtree $U\subseteq S$ is an automata-factor subtree if and only if the corresponding subtree of $T_{\mathbf t}$ is a factor subtree in the sense of Crawley--Boevey.
\item A subtree $V\subseteq S$ is an automata-image subtree if and only if the corresponding subtree of $T_{\mathbf t}$ is an image subtree in the sense of Crawley--Boevey.
\item An isomorphism $\Phi:U\xrightarrow{\sim}V$ is automata-label-preserving if and only if it is a quiver isomorphism satisfying
$
F_{\mathbf t}\circ\Phi=F_{\mathbf t}|_U.
$
\end{enumerate}

Consequently, $\mathscr O_{\mathcal A}(\mathbf t)$ is canonically in bijection with the set of Crawley--Boevey graph-map data from $(T_{\mathbf t},F_{\mathbf t})$ to itself.
\end{proposition}

\begin{proof}
Let $(x,y)\in\partial_SU$. By construction of $T_{\mathbf t}$, the condition $\widetilde{\ell}(x,y)\in Q_1$ means exactly that the boundary edge is oriented $x\to y$. Hence every boundary edge points out of $U$ if and only if no arrow of $T_{\mathbf t}$ enters $U$ from its complement. This is precisely the factor-subtree condition, proving \textup{(i)}. The same argument with $Q_1^{-1}$ shows that every boundary edge points into $V$ if and only if no arrow leaves $V$, proving \textup{(ii)}.

For \textup{(iii)}, suppose first that $\Phi$ is automata-label-preserving. If $a:x\to y$ is an arrow of $U$, then
$
\widetilde{\ell}(x,y)=F_{\mathbf t}(a)\in Q_1.
$
Since $\Phi$ preserves intrinsic syllables,
$
\widetilde{\ell}(\Phi(x),\Phi(y))
=
\widetilde{\ell}(x,y)\in Q_1.
$
Thus the edge joining $\Phi(x)$ and $\Phi(y)$ is oriented $\Phi(x)\to\Phi(y)$, so $\Phi$ is a quiver isomorphism. Moreover its arrow label is the same:
$
F_{\mathbf t}(\Phi(a))
=
\widetilde{\ell}(\Phi(x),\Phi(y))
=
\widetilde{\ell}(x,y)
=
F_{\mathbf t}(a).
$
The equality of vertex types gives the same compatibility on vertices. Hence $F_{\mathbf t}\circ\Phi=F_{\mathbf t}|_U$.

Conversely, if $\Phi$ is a quiver isomorphism satisfying $F_{\mathbf t}\circ\Phi=F_{\mathbf t}|_U$, then it preserves vertex types and the $Q$-label of every oriented edge. It therefore preserves the intrinsic syllable labelling, including the formal inverse labels obtained by traversing an edge in the opposite direction. Thus $\Phi$ is automata-label-preserving.

The final assertion follows immediately from \textup{(i)}--\textup{(iii)}.
\end{proof}

We can now identify the automata-induced notion of a tree brick with the representation-theoretic one.

\begin{theorem}[Automata--Crawley--Boevey brick correspondence]\label{thm:brick-criterion}
Let $\mathbf t\in L(\mathcal A_\Lambda)$, let $(T_{\mathbf t},F_{\mathbf t})$ be the associated tree datum, and let
$
M_{\mathbf t}:=M(T_{\mathbf t},F_{\mathbf t})
$
be the corresponding Crawley--Boevey tree module. Then the following are equivalent:
\begin{enumerate}[label=\textup{(\roman*)}]
\item $\mathbf t$ is an automata-induced tree brick;
\item the only Crawley--Boevey graph-map datum from $(T_{\mathbf t},F_{\mathbf t})$ to itself is
$
(T_{\mathbf t},T_{\mathbf t},\operatorname{id}_{T_{\mathbf t}});
$
\item $M_{\mathbf t}$ is a brick, that is,
$
\operatorname{End}_{\Lambda}(M_{\mathbf t})=K\,1_{M_{\mathbf t}}.
$
\end{enumerate}
Moreover,
$
\dim_K\operatorname{End}_{\Lambda}(M_{\mathbf t})
=
|\mathscr O_{\mathcal A}(\mathbf t)|.
$
\end{theorem}

\begin{proof}
By Proposition~\ref{prop:automata-graph-map-dictionary}, automata-induced self-overlaps of $\mathbf t$ are in canonical bijection with Crawley--Boevey graph-map data from $(T_{\mathbf t},F_{\mathbf t})$ to itself, and the trivial overlap corresponds to $(T_{\mathbf t},T_{\mathbf t},\operatorname{id}_{T_{\mathbf t}})$. Hence \textup{(i)} and \textup{(ii)} are equivalent.

By Crawley--Boevey's graph-map basis theorem \cite[\S 2]{CB89}, the graph maps indexed by these data form a $K$-basis of $\operatorname{End}_{\Lambda}(M_{\mathbf t})$. The datum $(T_{\mathbf t},T_{\mathbf t},\operatorname{id}_{T_{\mathbf t}})$ gives the identity endomorphism. Therefore the endomorphism algebra is one-dimensional if and only if this is the only graph-map datum. This proves the equivalence of \textup{(ii)} and \textup{(iii)}, and the same basis theorem gives
$
\dim_K\operatorname{End}_{\Lambda}(M_{\mathbf t})
=
|\mathscr O_{\mathcal A}(\mathbf t)|.
$
\end{proof}

As a consequence, Definition~\ref{def:automata-tree-brick} is not a new class of modules: it is an automata-theoretic presentation of precisely the brick tree modules among Crawley--Boevey tree modules.

\begin{corollary}\label{cor:tree-datum-automata-brick}
Let $(T,F)$ be a $\Lambda$-tree datum and choose any root $r\in T_0$. Let $\mathcal S_r$ be the rooted syllable tree obtained by rooting $T$ at $r$. Set $\mathbf t_r:=\Enc(\mathcal S_r)$. Then
\[
M(T,F)\text{ is a brick}\quad\Longleftrightarrow\quad\mathbf t_r\text{ is an automata-induced tree brick}.
\]
In particular, the right-hand condition is independent of the chosen root.
\end{corollary}

\begin{proof}
Theorem~\ref{thm:rooted-encoding} reconstructs $(T,F)$ from $\mathbf t_r$, so the equivalence follows from Theorem~\ref{thm:brick-criterion}. Independence of the root also follows from Proposition~\ref{prop:automata-brick-rerooting}.
\end{proof}

\begin{example}[An automata-induced brick and a non-brick]\label{ex:brick-nonbrick}
Let $Q$ be the oriented $2$-cycle
$
\begin{tikzcd}[column sep=huge]
1 \arrow[r,bend left=18,"\alpha"] & 2 \arrow[l,bend left=18,"\beta"]
\end{tikzcd}
$
and let $I=J^3=\langle\alpha\beta\alpha,\beta\alpha\beta\rangle$.

First consider the tree datum $T_2:\ x\xrightarrow{a}z,\ 
F(x)=1,\quad F(z)=2,\  F(a)=\alpha.$ Its ranked encoding is accepted by $\mathcal A_\Lambda$. The singleton $\{x\}$ is an automata-factor subtree and $\{z\}$ is an automata-image subtree, but they cannot be automata-label-preservingly isomorphic because their vertex types are different. There is therefore no non-trivial automata-induced self-overlap, and the accepted tree is an automata-induced tree brick. By Theorem~\ref{thm:brick-criterion}, $M(T_2,F)$ is a brick.

Now consider
$
T_3:\  x\xrightarrow{a}z\xrightarrow{b}y,
$
where $F(x)=F(y)=1$, $F(z)=2$, $F(a)=\alpha$ and $F(b)=\beta$. Since its only path of length two maps to $\beta\alpha\notin I$, this is again a tree datum and its encoding is accepted by $\mathcal A_\Lambda$. The singleton $\{x\}$ is an automata-factor subtree, while $\{y\}$ is an automata-image subtree. Since both vertices have type $1$, the map $x\mapsto y$ is automata-label-preserving. Thus
$
(\{x\},\{y\},x\mapsto y)
$
is a non-trivial automata-induced self-overlap. Hence the accepted tree is not an automata-induced tree brick, and Theorem~\ref{thm:brick-criterion} shows that $M(T_3,F)$ is not a brick. The corresponding graph endomorphism is
$$
v_x\longmapsto v_y,\qquad
v_z\longmapsto0,\qquad
v_y\longmapsto0.
$$
\end{example}
\section{Local colourings and alphabet compression}
\label{sec:compression}

The automaton of \S~\ref{sec:automaton} uses the full signed arrow alphabet $Q_1^{\pm}$. Inspired by the local-bijection construction for multi-entry inverse automata in \cite[Definition~4.11]{KuberSengupta2026}, we now show that the alphabet can be compressed while retaining all information relevant to tree data and bricks.

For $i\in Q_0$ put $Q_1^+(i)=\{\alpha:s(\alpha)=i\}$ and $Q_1^-(i)=\{\alpha:t(\alpha)=i\}$.

\begin{definition}\label{def:local-colouring}
Let $A$ be a finite set. A surjective map $c:Q_1\to A$ is a \emph{local colouring} if its restrictions to $Q_1^+(i)$ and to $Q_1^-(i)$ are injective for every $i\in Q_0$.
\end{definition}

Let $A^{-1}$ be a disjoint copy of $A$, put $A^{\pm}=A\sqcup A^{-1}$, and extend $c$ by $c^{\pm}(\alpha^{-1})=c(\alpha)^{-1}$. For $i\in Q_0$ set $\operatorname{Col}_c(i):=c^{\pm}(\Syl(i))$.

\begin{lemma}[Local reconstruction]\label{lem:local-reconstruction}
For every $i\in Q_0$, the restriction
\[
c_i^{\pm}:\Syl(i)\longrightarrow\operatorname{Col}_c(i)
\]
is a bijection.
\end{lemma}

\begin{proof}
Surjectivity is immediate. If two direct syllables in $\Syl(i)$ have the same colour, they are arrows with common source $i$ and hence are equal by local injectivity. If two inverse syllables have the same colour, their underlying arrows have common target $i$ and are again equal. A direct and an inverse syllable cannot have the same signed colour because $A$ and $A^{-1}$ are disjoint.
\end{proof}

Write $\lambda_i=(c_i^{\pm})^{-1}$. A \emph{$c$-coloured rooted tree} consists of a finite rooted tree $S$, a type map $\tau:S_0\to Q_0$, and labels $\bar\ell(x,y)\in A^{\pm}$ on parent--child edges such that $\bar\ell(x,y)\in\operatorname{Col}_c(\tau(x))$ and
\[
t\bigl(\lambda_{\tau(x)}(\bar\ell(x,y))\bigr)=\tau(y).
\]
Its \emph{lift} replaces each colour $\bar\ell(x,y)$ by the unique syllable $\lambda_{\tau(x)}(\bar\ell(x,y))$. We call the coloured tree $\Lambda$-admissible if its lift is $\Lambda$-admissible in the sense of Definition~\ref{def:admissible-syllable-tree}. Conversely, compression of a rooted syllable tree simply applies $c^{\pm}$ to every edge label.

For a $c$-coloured rooted tree, extend the signed colour labelling to both orientations of every edge by inversion. We use the terms \emph{automata-factor subtree} and \emph{automata-image subtree} with the same boundary definitions as in Definition~\ref{def:automata-factor-image}, replacing $Q_1$ and $Q_1^{-1}$ by $A$ and $A^{-1}$, respectively. An isomorphism between two coloured subtrees is \emph{automata-label-preserving} if it preserves vertex types and intrinsic signed colours. An admissible $c$-coloured rooted tree is called an \emph{automata-induced tree brick} if its only factor--image self-overlap in this sense is the trivial one.

\begin{theorem}[Compression of automata-induced tree bricks]
\label{thm:compression}
Compression and lifting are mutually inverse bijections between $\Lambda$-admissible rooted syllable trees and $\Lambda$-admissible $c$-coloured rooted trees. They commute with rerooting. Moreover, compression preserves and reflects automata-factor subtrees, automata-image subtrees, and label-preserving isomorphisms between subtrees. Consequently, it induces a bijection between automata-induced self-overlaps before and after compression, and hence preserves and reflects automata-induced tree bricks.
\end{theorem}

\begin{proof}
Lemma~\ref{lem:local-reconstruction} shows that a vertex type together with a signed colour recovers the original syllable uniquely. Hence compression and lifting are mutually inverse. Since $c^{\pm}$ commutes with inversion, the intrinsic signed colour attached to an oriented edge is independent of the chosen root, and therefore compression commutes with rerooting.

Let $U$ be a subtree. A boundary syllable is direct precisely when its compressed label lies in $A$, and it is inverse precisely when its compressed label lies in $A^{-1}$. Hence the automata-factor and automata-image boundary conditions are preserved and reflected by compression.

Now let $\Phi:U\to V$ be an isomorphism of subtrees preserving vertex types and compressed labels. If $(x,y)$ is an oriented edge of $U$, then the original syllables on $(x,y)$ and $(\Phi(x),\Phi(y))$ both belong to $\Syl(\tau(x))$
and have the same signed colour. By Lemma~\ref{lem:local-reconstruction}, these syllables are equal. Thus $\Phi$ is automata-label-preserving before compression. The converse is immediate.

It follows that compression induces a bijection between the automata-induced self-overlaps of a rooted syllable tree and those of its compressed coloured tree, with the trivial overlap corresponding to the trivial overlap. Therefore the original tree is an automata-induced tree brick if and only if its compression is an automata-induced tree brick.
\end{proof}

The automaton itself transports along the same bijection. Define a compressed ranked alphabet by
\[
\Sigma^c_{\Lambda,m}:=\{[i;a_1,\ldots,a_m]:a_j\in\operatorname{Col}_c(i),\ \lambda_i(a_1)\prec\cdots\prec\lambda_i(a_m)\},
\]
and put $\Sigma^c_\Lambda=\coprod_m\Sigma^c_{\Lambda,m}$. Every symbol has a unique lift
\[
[i;a_1,\ldots,a_m]\longmapsto[i;\lambda_i(a_1),\ldots,\lambda_i(a_m)]\in\Sigma_{\Lambda,m}.
\]
Use the same states and final states as $\mathcal A_\Lambda$ and define each compressed transition by first lifting its input symbol and then applying the transition of $\mathcal A_\Lambda$.

\begin{corollary}[Compressed recognition]\label{cor:compressed-recognition}
The resulting deterministic finite bottom-up automaton $\mathcal A_{\Lambda,c}$ recognizes exactly the ranked encodings of $\Lambda$-admissible $c$-coloured rooted trees. Under the lifting bijection, its accepted trees correspond bijectively to $L(\mathcal A_\Lambda)$, and automata-induced brickness is preserved.
\end{corollary}

\begin{proof}
At every vertex the compressed transition is, by definition, the original transition on the unique lifted symbol. Induction on the height therefore identifies the complete run of $\mathcal A_{\Lambda,c}$ with the run of $\mathcal A_\Lambda$ on the lifted tree. Apply Theorems~\ref{thm:recognition} and~\ref{thm:compression}.
\end{proof}

\subsection{Optimal number of colours}

Set
\[
\Delta(Q):=\max\left\{\max_{i\in Q_0}|Q_1^+(i)|,\ \max_{i\in Q_0}|Q_1^-(i)|\right\}.
\]
Form the bipartite multigraph $B_Q$ with vertex classes $Q_0^{\mathrm{out}}$ and $Q_0^{\mathrm{in}}$, and replace every arrow $\alpha:i\to j$ by an edge joining $i^{\mathrm{out}}$ to $j^{\mathrm{in}}$. A local colouring of $Q$ is exactly a proper edge-colouring of $B_Q$. Moreover, $\Delta(B_Q)=\Delta(Q)$.

\begin{theorem}[Optimal alphabet compression]\label{thm:optimal-colouring}
There exists a local colouring $c:Q_1\to A$ with $|A|=\Delta(Q)$, and no local colouring uses fewer colours. Consequently the signed arrow alphabet $Q_1^{\pm}$ can be replaced, without changing tree data or tree bricks, by a signed colour alphabet of size $2\Delta(Q)$; among local colourings this size is optimal.
\end{theorem}

\begin{proof}
By K\"onig's line-colouring theorem for bipartite multigraphs \cite[Theorem~1.4.18]{LovaszPlummer}, $B_Q$ has a proper edge-colouring with $\Delta(B_Q)=\Delta(Q)$ colours. This is a local colouring of $Q$. Conversely, at a vertex attaining the maximal out-degree or in-degree, all incident arrows must receive distinct colours, so every local colouring uses at least $\Delta(Q)$ colours. The assertion for signed alphabets follows by adjoining formal inverses of the colours.
\end{proof}

\begin{example}[Compression to a binary signed alphabet]\label{ex:compression-cycle}
Let $Q$ be the directed $4$-cycle
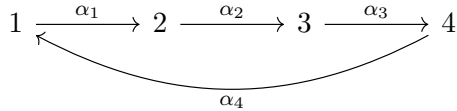
\[
\begin{tikzcd}[column sep=large]
1 \arrow[r,"\alpha_1"] & 2 \arrow[r,"\alpha_2"] & 3 \arrow[r,"\alpha_3"] & 4 \arrow[lll,bend left=28,"\alpha_4"]
\end{tikzcd}
\]
and take $I=J^5$. Then $\Lambda=KQ/J^5$ is finite-dimensional and $\Delta(Q)=1$. Hence the constant colouring $c(\alpha_i)=a$ is local. Although the original signed alphabet has eight symbols $\alpha_i^{\pm1}$, the compressed signed alphabet is only $\{a,a^{-1}\}$. No ambiguity is introduced: at a vertex $i$, the direct colour $a$ lifts to the unique arrow leaving $i$, while $a^{-1}$ lifts to the inverse of the unique arrow entering $i$. Thus the vertex type supplies precisely the information lost by forgetting the arrow name.
\end{example}

\section{Is the language of tree bricks regular?}
\label{sec:regularity}

Theorem~\ref{thm:recognition} shows that the class of rooted tree data over $\Lambda$ is a regular tree language. Theorem~\ref{thm:brick-criterion} singles out the bricks inside this language by a global condition: there must be no non-trivial factor--image self-overlap.

For a $\Lambda$-admissible rooted syllable tree $\mathcal S$, write $(T_{\mathcal S},F_{\mathcal S})$ for the tree datum reconstructed by Theorem~\ref{thm:rooted-encoding}. Define
\[
\mathsf{Brick}_\Lambda:=\{\Enc(\mathcal S)\in L(\mathcal A_\Lambda):M(T_{\mathcal S},F_{\mathcal S})\text{ is a brick}\}.
\]
By Corollary~\ref{cor:rerooting} and Theorem~\ref{thm:brick-criterion}, membership is invariant under rerooting.

\begin{question}\label{que:regular-bricks}
Is $\mathsf{Brick}_\Lambda$ a regular tree language for every finite-dimensional zero-relation algebra $\Lambda$?
\end{question}

The automaton $\mathcal A_\Lambda$ only checks local compatibility together with forbidden paths of bounded length, so finite memory is immediate. Brickness is different: a witness to non-brickness consists of a factor subtree and an image subtree which may be arbitrarily large and may lie in distant branches, together with a label-preserving isomorphism between them. Thus a putative finite automaton for non-brickness would have to encode enough information about partial self-overlaps to know whether two such occurrences can be completed.

One possible route is to search for a finite quotient of the possible partial overlap types during a bottom-up run. If the non-brick language were regular, closure of regular tree languages under complementation would give regularity of $\mathsf{Brick}_\Lambda$ as well; see \cite{TATA}. Equivalently, one may ask whether the factor--image self-overlap condition admits a monadic second-order description on the ranked encodings, using the standard equivalence between recognizable finite tree languages and monadic second-order definability \cite{TATA}. The present paper leaves this as the natural next automata-theoretic problem.

\section*{Statement on the use of artificial intelligence.}
Generative artificial intelligence tools were used during the preparation of this manuscript. Their use included assistance with organizing the structure of the paper, developing and refining the exposition, reformulating definitions and statements, suggesting intermediate lemmas and proof strategies, checking the internal consistency of notation, producing and revising \LaTeX{} code, and discussing examples and possible directions related to the automata-theoretic formulation developed here. AI tools were also used as a conversational aid in exploring how the classical theory of tree modules and graph maps could be recast in the language of finite tree automata and in presenting the comparison between automata-induced tree bricks and the representation-theoretic notion of a brick. The mathematical content generated or suggested through these interactions was not accepted automatically. The author independently checked the definitions, theorem statements, proofs, examples, bibliographic references, and logical dependencies appearing in the final manuscript, and modified or discarded AI-generated material whenever necessary. Responsibility for the correctness of all mathematical arguments and for the final form and claims of the paper rests entirely with the author.

\bibliographystyle{alpha}
\bibliography{main}

\end{document}